%% file: main.tex
\documentclass[12 pt, psamsfonts]{amsart}
\usepackage{amsmath}
\usepackage{amsthm}
\usepackage{amsfonts}
\usepackage{amssymb}
\usepackage{float}
\usepackage{pgf,tikz}
\usepackage{tikz-cd}
\usepackage{mathrsfs}
\usetikzlibrary{arrows}
\usepackage{eucal}
\usepackage{graphicx}
\usepackage{multirow}
\usepackage[all,knot]{xy}
\xyoption{arc}

\DeclareMathOperator{\Aut}{Aut}
\newcommand{\cH}{\mathcal{H}}
\newcommand{\D}{\mathcal{D}}
\newcommand{\CC}{\mathcal{C}}

\newcommand{\End}{{\rm End}}
\newcommand{\Rep}{{\rm Rep}}

\newcommand{\GL}{{\rm GL}}
\newcommand{\SL}{{\rm SL}}
\newcommand{\im}{\mathrm{i}}
\newcommand{\ot}{\otimes}
\newcommand{\B}{\mathcal{B}}

\newcommand{\A}{\mathcal{A}}

\newcommand{\om}{\omega}

\newcommand{\Hom}{{\rm Hom}}

\newcommand{\T}{\mathcal{T}}

\newcommand{\F}{\mathbb{F}}

\newcommand{\Z}{\mathbb{Z}}

\newcommand{\Q}{\mathbb{Q}}
\newcommand{\C}{\mathbb{C}}
\newcommand{\R}{\mathbb{R}}

\newcommand{\FPdim}{{\rm FPdim}}

\newcommand{\ppg}[1]{{\textcolor{red}{#1}}}
\newcommand{\er}[1]{{\textcolor{blue}{#1}}}

\newcommand{\ignore}[1]{}  
\newcommand{\scrA}{\mathscr{A}}
\newcommand{\legendre}[2]{\genfrac{(}{)}{}{}{#1}{#2}}

\newtheorem{theorem}{Theorem}[section]
\newtheorem*{theorem*}{Theorem}

\newtheorem{prop}[theorem]{Proposition}
\theoremstyle{definition}

\newtheorem{example}[theorem]{Example}
\newtheorem{principle}[theorem]{Principle}
\newtheorem{defn}[theorem]{Definition}

\theoremstyle{remark}

\newtheorem{problem}[theorem]{Problem}
\numberwithin{equation}{section}

\calclayout

\title[$\B_n$ representations from twisted tensor products]{Braid group representations from twisted tensor products of algebras}
\author{Paul Gustafson$^1$, Andrew Kimball$^2$, Eric C. Rowell$^2$, Qing Zhang$^2$}
\date{\today}

\address{$^1$Department of Electrical Engineering\\
Wright State University \\
Dayton, OH 45435\\
U.S.A.}
\email{paul.gustafson@wright.edu}
\address{$^2$Department of Mathematics\\
    Texas A\&M University\\
    College Station, TX 77843-3368\\
    U.S.A.}
\email{amkimball1@math.tamu.edu, rowell@math.tamu.edu, zhangqing@math.tamu.edu}

\thanks{The authors gratefully acknowledge support under USA NSF grant DMS-1664359.  We also thank C. Galindo, JM Landsberg, Z. Wang and S. Witherspoon for valuable insight. ER is also partially supported by a Texas A\&M Presidential Impact Fellowship.}

\begin{document}
\begin{abstract}
We unify and generalize several approaches to constructing braid group representations from finite groups, using iterated twisted tensor products. Our results hint at a relationship between the braiding on the $G$-gaugings of a pointed modular category $\CC(A,Q)$ and that of $\CC(A,Q)$ itself.
\end{abstract}
\maketitle

\section{Introduction}

Braid group representations are plentiful: for example,
from any object $X$ in a braided fusion category $\CC$ one obtains a sequence of braid group representations $\rho^X:\B_n\rightarrow \End(X^{\ot n})$.
Braided vector spaces $(R,V)$ \cite{AS} (i.e., matrix solutions to the Yang-Baxter equation) are also a very rich source, as are families of finite dimensional quotients of the braid group algebra $\F(\B_n)$, such as
Hecke algebras, Temperley-Lieb algebras \cite{jones86} and BMW-algebras \cite{BW,M}.  That the braided fusion category construction essentially supersedes these sources is well-known, but explicit matrices for the
generators are not easy to come by--one typically needs the associativity
constants ($6j$-symbols or $F$-matrices) in addition to the $R$-symbols, 
and these are only available for a few families of categories.  Generally,
the irreducible representations of $\B_n$ are only classified for dimensions
at most $n$ \cite{formaneketal} and for $\B_3$ for dimensions up to $5$ \cite{tubawenzl1}.  Braided vector spaces $(R,V)$ are only classified for $\dim(V)=2$ \cite{Hiet}.  

In this article we outline an approach to finding families of braid group representations from twisted tensor products of algebras. 
We motivate our approach with the following two ``proof of principle" examples.

In \cite{GoldschmidtJones} the following representations of the braid group $\B_n$ were described: set $q=e^{2\pi i/m}$ for $m$ odd and let $\A_n(\Z_m)$ be the $\C$-algebra generated by $u_i$ for $1\leq i\leq n-1$ satisfying:
\begin{itemize}
    \item $u_i^m=1$
    \item $u_iu_{i+1}=q^{2}u_{i+1}u_i$
    \item $u_iu_j=u_ju_i$ for $|i-j|\neq 1$
\end{itemize}
The algebra $\A_n(\Z_m)$ is denoted by $ES(m,n-1)$ in \cite{Jones89} and by $T_n^m(q)$ in \cite{RWQT}.  Our perspective is to regard $\A_n(\Z_m)$ as an iterated twisted tensor product of the group algebra $\C[\Z_m]$: 
$$\A_n(\Z_m)=\C[\Z_m]\ot_\vartheta\C[\Z_m]\ot_\vartheta\cdots\ot_\vartheta \C[\Z_m],$$ where $\vartheta$ is the twisting map corresponding to the second relation above.  Defining $$\rho_n(\sigma_i)=R_i:=\frac{1}{\sqrt{m}}\sum_{j=0}^m q^{j^2}u_i^j$$  we obtain a representation $\B_n\rightarrow \A_n(\Z_m)$.  These representations are known to have finite image, see \cite{GR,GoldschmidtJones}.
 Moreover, we may obtain a matrix representation by defining $U\in\End(\C^m\otimes\C^m)$ by $U(\mathbf{e}_i\otimes \mathbf{e}_j)=q^{j-i}\mathbf{e}_{i+1}\ot \mathbf{e}_{j+1}$ and assigning $$u_i\mapsto Id^{\ot i}\ot U\ot Id^{\ot n-i-1}.$$ From this representation one obtains a braided vector space as $R:=\rho_2(\sigma_1)$ on $V=\C^m$.

Another example of braid group representations related to twisted tensor products of algebras is found in \cite{RQT} (due to Jones) where the quaternion group $Q_8$ appears.  For $1\leq i\leq n-1$ let $\A_n(Q_8)$ be the algebra generated by $u_i,v_i$ satisfying:
\begin{enumerate}
\item $u_i^2=v_i^2=-1$ for all $i$,
\item $[u_i,v_j]=-1$ if $|i-j|<2$,
\item $[u_i,u_j]=[v_i,v_j]=1$,
\item $[u_i,v_j]=1$ if $|i-j|\geq2$.

\end{enumerate}

Although $\A_n(Q_8)$ is not, strictly speaking, a twisted tensor product of group algebras we nonetheless obtain braid group representations via
$$\sigma_i\mapsto (1+u_i+v_i+u_iv_i).$$

In this article we initiate the general problem of finding braid group representations in twisted tensor products of (group) algebras, unifying the two examples just outlined.  

Our study is not just motivated by idle curiosity.  
In the last section we explore some relationships between these twisted tensor products of group algebras and $G$-gaugings of pointed modular categories, laying the groundwork towards understanding braid group representations associated with weakly group theoretical modular categories and the property $F$ conjecture.  This conjecture \cite{NR} predicts that the braid group images obtained from any weakly integral braided fusion category $\CC$ have finite image, i.e., $\CC$ has property $F$.  By taking Drinfeld centers one may reduce this conjecture to the case where $\CC$ is a modular category.  The property $F$ conjecture has been verified for many classes of braided fusion categories: for example, group-theoretical categories \cite{ERW}, quantum group categories \cite{jones86,FLW,LR,RQT,RWQT}, and certain metaplectic categories \cite{GRRTohoku}.

The paper is organized as follows: in section \ref{prelim} we set down the general framework for our problem, which is explicitly described and analyzed in section \ref{main problem}.  We carry out several case studies for both abelian and non-abelian cases in section \ref{case studies} while the connections to categories obtained by gauging symmetries of pointed modular categories are speculated upon in section \ref{categorical connections}, followed by a short section of conclusions.  An appendix contains some MAGMA code for some explicit examples.

\section{Preliminaries}\label{prelim}
We first describe the general algebraic ingredients for the problem we are interested in.

\subsection{Twisted Tensor Products}
The treatment of twisted tensor products in \cite{cap} is most suitable for our purposes:
Let $A$ and $B$ be $\F$-algebras with multiplication maps $\mu_A,\mu_B$ respectively, and a map $\vartheta:B\otimes A\rightarrow A\otimes B$, i.e. $\vartheta$ is $\F$-linear map with $\vartheta(b\ot 1)=(1\ot b)$ and $\vartheta(1\ot a)=(a\ot 1)$.  The map $\mu_\vartheta:  A\ot B\rightarrow A\ot B$ defined by $$\mu_\vartheta=(\mu_A\ot\mu_B)\circ(Id_A\ot\vartheta\ot Id_B)$$ defines an associative multiplication if and only if:
\begin{equation}\label{tautwist}
    \vartheta\circ (\mu_B\ot \mu_A)=\mu_\vartheta\circ(\vartheta\ot\vartheta)\circ (Id_B\ot\vartheta\ot Id_A).
\end{equation} The corresponding algebra, denoted $A\ot_\vartheta B$ will be called a \textbf{twisted tensor product} of $A$ and $B$, and the map $\vartheta$ will be called a unital twisting map.  We shall be most interested in the case where $A=B$.

To iterate this process, we rely on results of \cite{iterated}.
Given 3 algebras $A,B$ and $C$ and unital twisting maps $\vartheta_1:B\ot A\rightarrow A\ot B$, $\vartheta_2:C\ot B\rightarrow B\ot C$ and $\vartheta_3:C\ot A\rightarrow A\ot C$ each of which satisfy eqn. (\ref{tautwist}) one can define two maps $T_1=(Id_A\ot \vartheta_2)\circ (\vartheta_3\ot Id_B)$ on $C\ot(A\ot_{\vartheta_1}B)$ and $T_2=(\vartheta_1\ot Id_C)\circ(Id_B\ot \vartheta_3)$ on $(B\ot_{\vartheta_2}C)\ot A$ which are potentially unital twisting maps.  \cite[Theorem 2.1]{iterated} show that these are both unital twisting maps if and only if the compatibility condition:
\begin{equation}\label{iteratedcompat}
    (Id_A \ot \vartheta_2)  (\vartheta_3 \ot Id_B) (Id_C \ot \vartheta_1)=(\vartheta_1 \ot Id_C ) (Id_B\ot \vartheta_3)  (\vartheta_2 \ot Id_A)
\end{equation} is satisfied. Moreover, the two \textbf{iterated twisted tensor products} $(A\ot_{\vartheta_1} B)\ot_{T_1}C$ and $A\ot_{T_2}(B\ot_{\vartheta_2}C)$ constructed from these twisting maps are isomorphic algebras.  One may inductively define twisted tensor products for any number of algebras $A_i$ provided the analogous compatibility conditions are satisfied.  Again, we will be especially interested in the case where $A=A_i=A_j$, and $\vartheta=\vartheta_{i,i+1}:A_{i+1}\ot A_i\rightarrow A_{i}\ot A_{i+1}$ for adjacent copies of $A$ and $\sigma=\vartheta_{i,j}:A_j\ot A_i\rightarrow A_i\ot A_j$ for $|i-j|>1$ is the usual flip map $\sigma(a\ot b)=b\ot a$.  In fact, for all of our examples we will have $\vartheta(a\ot b)=\tau(a,b)b\ot a$ for some function $\tau:A\ot A\rightarrow \F$.  One then easily sees that \ref{iteratedcompat} is satisfied:
$$(Id \ot \vartheta)(\sigma \ot Id) (Id \ot \vartheta)(a\ot b\ot c) \quad\text{and}\quad(\vartheta \ot Id) (Id\ot \sigma)  (\vartheta \ot Id) (a\ot b\ot c)$$ 
are both equal to $\tau(a,b)\tau(b,c)(c\ot b\ot a)$.  Moreover, condition $(\ref{tautwist})$ and unitality are equivalent to $\tau:A\ot A\rightarrow \F$ being a bihomomorphism of $\F$-algebras: $$\tau(a_1a_2,b_1b_2)=\tau(a_2,b_1)\tau(a_1,b_1)\tau(a_2,b_1)\tau(a_2,b_2),$$  and unitality implies $\tau(1,a)=\tau(a,1)=1$, while bilinearity is immediate.

\subsection{Braid group representations and property $F$}
Our goal is to study families of representations of the braid group $\B_n$.
In particular we are interested in representations that are related in the following way:
\begin{defn}\label{seq}
 An indexed family of complex $\B_n$-representations $(\rho_n,V_n)$ is a \emph{sequence of braid representations} if there exist injective algebra homomorphisms $\iota_n:\C\rho_n(\B_n)\rightarrow \C\rho_{n+1}(\B_{n+1})$ such that the following diagram commutes:
$$\xymatrix{\C\B_n\ar[r]\ar@{^{(}->}[d] & \C\rho_n(\B_n)\ar@{^{(}->}[d]^{\iota_n} \\ \C\B_{n+1}\ar[r] & \C\rho_{n+1}(\B_{n+1})}$$
where the left-hand side of the square is induced by the inclusion $\B_n \hookrightarrow B_{n+1}$ given by $\sigma_i \mapsto \sigma_i$.
\end{defn}

Our examples are typically of the following form: let $1\in \scrA_1\subset \scrA_2\subset \cdots\subset \scrA_n\subset\cdots$ be a tower of finite dimensional semisimple algebras, and $\rho_n:\C\B_n\rightarrow \scrA_n$ algebra homomorphisms that respect the inclusions $\scrA_n\subset\scrA_{n+1}$.  Then the canonical faithful representation of $\scrA_n$ provides a sequence of representations.

For example,  we obtain a sequence of $\B_n$-representations from any braided vector space $(R,V)$ i.e., an invertible operator $R\in\Aut(V^{\ot 2})$ that satisfies the Yang-Baxter equation
$$(R\ot I)(I\ot R)(R\ot I)=(I\ot R)(R\ot I)(I\ot R)\in\Aut(V^{\ot 3}).$$ Explicitly we have $\B_n\rightarrow \Aut(V^{\ot n})$ via $\sigma_i\rightarrow Id_V^{\ot i-1}\ot R\ot Id_V^{\ot n-i-1}$.

Other standard examples come from the Temperley-Lieb, Hecke and BMW-algebras mentioned in the Introduction.

Some conjectures on the images of such representations are found in \cite{RW,GHR,GR}. For example it is an open question whether unitary braided vector spaces have virtualy abelian images, but there is strong evidence that this is so.

\section{Twisted Tensor Products of Algebras and Yang-Baxter Operators}\label{main problem}
The problem that we propose to study is the following: 
\begin{problem}
Find and classify braid group representations inside twisted tensor products of (group) algebras, generalizing the well-known Gaussian solutions.
\end{problem}

\subsection{Twisted tensor products of group algebras}

First we describe the twisted tensor products of (group) algebras we will study.
  Fix a finite group $G$ and  $q\in U(1)$.  On the group algebra $\Q(q)[G]$ we would like to find a unital twisting map $\vartheta:\Q(q)[G]\ot \Q(q)[G]\rightarrow\Q(q)[G]\ot \Q(q)[G]$ such that $\vartheta(g_1\ot h_2)=\tau(g,h)h_2\ot g_1$ on basis elements where $\tau(g,h)=q^{\alpha(g,h)}$.  We use subscripts to emphasize that $g_1\in G$ is a basis element of the first factor and $h_2\in G$ is a basis element of the second factor.  If $\vartheta$ is a unital twisting map then $\tau(g,h):G\times G\rightarrow U(1)$ is a bicharacter of $G$.  Since $q^{\alpha(g^k,h)}=q^{k\alpha(g,h)}$ we assume that $q$ is an $m$th root of unity (where $m\mid exp(G)$), and that $\alpha:G\times G\rightarrow \Z_m$ is a bihomomorphism.  Now define $\sigma:\Q(q)[G]\ot \Q(q)[G]\rightarrow\Q(q)[G]\ot \Q(q)[G]$ to be the usual flip map $g_1\ot h_2=h_2\ot g_1$. It is routine to check that (\ref{iteratedcompat}) is satisfied by $\vartheta$ and $\sigma$.

With these verifications we can define a finite dimensional semisimple algebra $\A_n(G,\tau)$ as an iterated twisted tensor product of $\C[G]$ follows: as a $\Q(q)$ vector space $\A_n(G,\tau)=\Q(q)[G]^{\ot n-1}$.  For each $1\leq i\leq n-1$ and $g\in G$ we define elements $g_i=1^{\ot i-1}\ot g\ot 1^{\ot n-i-2}$.  We can then dispense with the $\ot$ symbol altogether, and write monomials as $g^{(i_1)}_1\cdots g^{(i_{n-1})}_{n-1}$ where $g^{(i_j)}\in G$. The multiplication on $\A_n(G,\tau)$ have the following straightening rules on the generators $g_i$: $$g_ih_j=\begin{cases} h_jg_i, & |i-j|>1 \\ q^{\pm\alpha(g,h)}h_{i\pm1}g_i & j=i\pm1
\\
(gh)_i &j=i\end{cases}$$
where the  bihomomorphism  $\alpha:G\times G\rightarrow \Z_m$ determines $\tau$.  The following is presumably well-known but can be proved directly using classical techniques, which we provide for the reader's amusement.
\begin{prop}
The algebra $\A_n(G,\tau)$ is semisimple of dimension $|G|^{n-1}$ over $\Q(q)$.
\end{prop}

\begin{proof}
Notice that the set of monomials in normal form $M:=\{q^\ell g^{(i_1)}_1\cdots g^{(i_{n-1})}_{n-1}: g^{(i_j)}\in G,\ell\in\Z_m\}$ form a basis for $\A_n(G,\tau)$ over $\Q$ since $\A_n(G,\tau)$ is $\Q(q)[G]^{\ot n-1}$ as a vector space.   To show that $\A_n(G,\tau)$ is semisimple, let $X\subset\A_n(G,\tau)$ be a submodule, and $\pi:\A_n(G,\tau)\rightarrow X$ any vector space projection.  Note that any $\textbf{t}\in M$ has an inverse in $M$, since the straightening rules allow us to write $\textbf{t}^{-1}$ in the normal form of $M$.  We then use the standard averaging trick to find an $\A_n(G,\tau)$-module projection onto $X$:
$$T_\pi(y):=\frac{1}{m|G|^{n-1}}\sum_{\textbf{t}\in M}\textbf{t}\pi(\textbf{t}^{-1}y).$$  One can readily check that $T_\pi$ is a surjective $\A_n(G,\tau)$-module homomorphism so that the kernel of $T_\pi$ provides a direct complement to $X$ in $\A_n(G,\tau)$, proving that $\A_n(G,\tau)$ is semisimple.
\end{proof}

\subsubsection{Connection to group extensions}
We also note the following alternative construction of the algebra $\A_n(G, \tau)$ via group extensions.  In this section, we show that, as a $\mathbb{Q}$-algebra, $\A_n(G, \tau)$ is isomorphic to the group algebra over $\mathbb{Q}$ of a central extension of $G^{n-1}$.

More concretely, let $G$ be a finite group and $m \mid \mathrm{exp}(G)$. Let $\alpha: G \times G \to \mathbb{Z}_m$ be a bihomomorphism.  For $n \ge 2$, let $c : G^n \times G^n \to \mathbb{Z}_m$ be the bihomomorphism defined by
$$c(g, h) = -\sum_{i=1}^{n-1} \alpha(h_i, g_{i+1}).$$
Since $c$ is a bihomomorphism, it satisfies the 2-cocycle condition.  We define ${G^{\times_\alpha n}}$ to be the central extension of $G^n$ corresponding to the 2-cocycle $c \in  Z^2(G^n, \mathbb{Z}_m)$.

\begin{prop}
Let $q$ be a primitive $m$-th root of unity. There is an isomorphism of $\mathbb{Q}$-algebras $$\A_{n+1}(G,\tau) \cong \mathbb{Q}(G^{\times_\alpha n}),$$ where $\tau \colon \Q(q)[G]\ot \Q(q)[G]\rightarrow\Q(q)[G]\ot \Q(q)[G]$ is the same twisting map defined above, i.e. $\tau(g_1\ot h_2)=q^{\alpha(g,h)}h_2\ot g_1$ on basis elements.
\end{prop}
\begin{proof}
We represent $G^{\times_{\alpha} n}$ as the set $\mathbb{Z}_m \times G^n$ where the multiplication is given by 
$$(x \times g) \cdot (y \times h) = (c(g,h) + x + y) \times gh.$$
Let $\phi : \A_{n+1}(G,\tau) \to \mathbb{Q}(G^{\times_\alpha n})$ be the $\mathbb{Q}$-linear bijection defined by $q^j g_1 \otimes \cdots \otimes g_{n-1} \mapsto j \times (g_1, \ldots g_{n-1})$.  To see that $\phi$ is an algebra map, we need to verify that it preserves the straightening relations.  If $i - j \neq 1$, we have $c(\phi(g_i), \phi(h_j)) = 1$.  If $i - j =  1$, we have $c(\phi(g_i), \phi(h_j)) = \alpha(h, g^{-1})$.  Thus, 
\begin{align*}
[\phi(g_i), \phi(h_{i+1})] & = \phi(g_i) \phi(h_{i+1}) \phi(g_i^{-1}) \phi(h_{i+1}^{-1}) \\
& =  \phi(g_i) (\alpha(g,h) \times (e, \ldots, e, g_i^{-1}, h_{i+1}, e, \ldots, e)) \phi(h_{i+1}^{-1}) \\ 
& = \alpha(g,h) \times e \\
& = \phi\left(q^{\alpha(g,h)} \right).
\end{align*}
Thus, the straightening relations are preserved by $\phi$. It follows that $\phi$ is an isomorphism of $\mathbb{Q}$-algebras.
\end{proof}


\subsection{Braid group representations}
    
    The second part of the problem is to look for and classify the representations of the braid group inside the algebras $\A_n(G,\tau)$.  Fix a pair $(G,\alpha)$ where $G$ is a finite group and $\alpha:G\times G\rightarrow \Z_m$ is a bihomomorphism and the corresponding twisted tensor power $\A_n(G,\tau)$.  We are interested in finding invertible:
$$r=\sum_{g\in G}f(g)g\in\C[G]$$
so that for $i=1,2$ the $r_i:=\sum_{g\in G}f(g)g_i\in\A_3(G,\tau)\ot_{\Q(q)} \C$  satisfy the braid relation $r_1r_2r_1=r_2r_1r_2$. We shall call such solutions $r$ \textbf{$\A(G,\tau)$-Yang-Baxter operators (YBOs)}. Since the braid equation may be written as a linear combination of monomials $g^{(i_1)}_1g^{(i_{2})}_2\in \A_3(G,\tau)$ with coefficients in $\Q(q)[x_1,\ldots,x_{|G|}]$ where $x_i:=f(g^{(i)})$, we may assume that the function $f$ takes values in $\overline{\Q(q)}=\overline{\Q}$ i.e., the algebraic closure of $\Q(q)$.  For the sake of notation we will usually just consider scalars in the complex field $\C$.

This should be compared with the problem of finding Yang-Baxter operators on a vector space $V$.  In our case we seek invertible $r\in\C[G]$ so that $\rho_n(\sigma_i)=r_i$ is a homomorphism $\rho_n:\B_n\rightarrow\A_n(G,\tau)$ (suitably complexified). As $\A_n(G,\tau)\ot\C$ is a finite dimensional semisimple $\C$-algebra, one can obtain $\B_n$-representations by pull-back on any $\A_n(G,\tau)\ot\C$-module. For example, one might use the regular representation to get a sequence of braid group representations $(\rho_n,V_n)$, as defined in \cite{RW}.  However one cannot, in general, turn such a homomorphism into a solution to the Yang-Baxter equation.  There is one situation where one can perform such a transformation: if the sequence of braid group representations $(\rho_n,V_n)$ is \emph{localizable} in the sense of \cite{RW}.  For example, suppose $\A_n(G,\tau)$ has a representation $\vartheta_n$ of the form $V^{\ot n}$ with $\vartheta(g_i)$ acting locally: $$\vartheta_n(g_i)(v_1\ot\cdots v_{i}\ot v_{i+1}\ot\cdots\ot v_n)=(v_1\ot\cdots \vartheta(g)(v_{i}\ot v_{i+1})\ot\cdots\ot v_n)$$ where $\vartheta:G\rightarrow \Aut(V^{\ot 2})$ is a $G$-representation.  Then $\vartheta(r)$ will be a Yang-Baxter operator, and $\vartheta_n(r_i)=(Id_V)^{\ot i-1}\ot \vartheta(r)\ot (Id_V)^{\ot n-i-1}$ is a localization of the corresponding braid group representation.

We may also put a $*$-structure on $\A_n(G,\tau)$ as follows: define $g_i^*=g_i^{-1}$ and $q^*=1/q=\overline{q}$ and then extend to an antiautomorphism on products and linearly on sums in the usual way.  This makes $\A_n(G,\tau)$ a $*$-algebra.  In this way we can discuss unitary $\A(G,\tau)$-YBOs, as those $r$ with $r^*r=1$.  

\subsection{Equivalence Classes of $\A_n(G,\tau)$}
For a fixed $G$, different choices of $\tau$ give isomorphic algebras.  We identify a few of these isomorphisms in order to reduce the complexity of our main goal.

One equivalence of $\A_n(G,\tau)$ comes from the choice of Galois conjugates of $q$.  For $(s,m)=1$ defining $\tau^s(g,h)=q^{s\alpha(g,h)}$ obviously gives us $\A_n(G,\tau)\cong\A_n(G,\tau^s)$ by Galois conjugation. 

Another equivalence comes from automorphisms of $G$.
If $\psi\in\Aut(G)$ then $\beta(g,h):=\alpha(\psi(g),\psi(h))$ gives us a new $\tau^\psi(g,h)=\tau(\psi(g),\psi(h))$ and gives us $\A_n(G,\tau)\cong\A_n(G,\tau^\psi)$. 

An important problem is to understand the orbits under these actions. The Galois symmetry amounts to replacing $\alpha$ with $s\alpha$.   Now observe that since $\alpha:G\times G\rightarrow \Z_m$ has abelian co-domain, it is determined by its values on the abelianization $G_{ab}:=G/[G,G]$. So we may assume $G=A$ is abelian for the purposes of determining the orbits.  It is clear that the bicharacters $A\times A\rightarrow U(1)$ for a finite abelian group $A$ form an abelian group under pointwise addition.  In fact, this group is isomorphic to $\Hom(A,A^*)$ where $A^*=\Hom(A,U(1))$ is the group of characters.  Indeed, if $\chi:A\times A\rightarrow U(1)$ is a bicharacter then define $F_\chi\in\Hom(A,A^*)$ by $F_\chi(a)(b)=\chi(a,b)$.  Since $\chi$ is a bicharacter $F_\chi$ is a $\Z$-module map with values in $\Hom(A,U(1))$.  Since $f\in\Hom(A,A^*)$ determines a unique bicharacter $\chi_f(a,b)=f(a)(b)$, the map $\chi\mapsto F_\chi$ is clearly a bijection, and $F_\chi+F_\eta=F_{\chi+\eta}$.  As the bihomomorphisms $\alpha:G\times G\rightarrow \Z_m$ are in one-to-one correspondence with bicharacters, this determines all such bihomomorphisms $\alpha$.   For elementary abelian $p$-groups $A=(\Z_p)^k$ a bihomomorphism to $\Z_p$ can be represented as a $k\times k$ matrix $X$ with $i,j$ entry $\alpha(e_i,e_j)\in\Z_{p}$ where $e_i$ is the generator $(0,\ldots,0,1,0,\ldots,0)$ of the $i$th factor.  That is, $\alpha(g,h)=g^TXh$ where we identify $g\in A$ with column vectors. Of course an automorphism $\Psi\in\Aut(A)\cong\GL_k(\Z_p)$ as well, so we may sweep out orbits of $\alpha$ under $\Aut(G)$ as $\Psi^T X\Psi$, since $\alpha(\Psi(g),\Psi(h))=g^T\Psi^TXh$.  Although we will not need it in what follows, one can handle general abelian $p$-groups in a similar way by identifying $\Z_{p^a}$ with a subgroup of $\Z_{p^b}$ for $a\leq b$.  Even more generally, bihomomorphisms on a finite abelian group can be factored by restricting to $p$-Sylow subgroups.

\subsubsection{Forbidden Symmetries}
Notice we have not considered applying \emph{different} automorphisms of $G$ to each factor, as this will conflict with our goal of finding $\A(G,\tau)$-YBOs--they will not have a uniform description, i.e.,  the form of $r_i$ not be independent of $i$.  Even ignoring this constraint, if we apply different automorphism to each tensor factor of $\A_n(G,\tau)$ the twisting will no longer be uniform across the iterated twisted tensor product.  However, in the following special cases, uniformity is preserved.  

\begin{prop}\label{smith form}  Suppose $G=\mathbb{Z}_m^k$ for an odd integer $m$ and $\tau(x,y)=q^{\alpha(x,y)}$ for a non-degenerate \emph{symmetric} or \emph{skew-symmetric} bihomomorphism $\alpha:G\times G\rightarrow \Z_p$.  Then there is an isomorphism of algebras $\A_n(G,\tau)\cong\A_n(G,\chi)$, where $\chi(x,y)=q^{x^Ty}$.
\end{prop}
\begin{proof}
Observe that for $G =\mathbb{Z}_p^k$ with $p$ odd, nondegenerate bihomomorphisms $\alpha: G \times G \to \mathbb{Z}_m$ are the same as nondegenerate bilinear forms, all of which are of the form $\alpha(x,y) = x^T A y$ for some matrix $A \in GL_k(\Z_m)$.   

First assume $\alpha(x,y) = x^T S y$ a non-degenerate symmetric bilinear map $G \times G \to \Z_m$, let  $A,B \in GL_k(\Z_m)$ be such that $I=ASB$ (the Smith normal form of $S)$.  Let $\eta(x,y):= x^T y$ be the corresponding twist.  We claim that the map $\phi: G^{n-1} \to G^{n-1}$ defined by 
\[
\phi(g_i) = \begin{cases}
((A^T)^{-1}g)_i, & i \text{ odd} \\
(B^{-1}g)_i, & i \text{ even}
\end{cases}
\]
induces an isomorphism $\A_n(G, \tau) \cong \A_n(G, \eta)$.  
Indeed, for $x,y \in G$, we have 
$$\eta((A^{T})^{-1} x, B^{-1}y) = ((A^T)^{-1} x)^T (A S B) (B^{-1} y)  = x^T S y.$$ 
On the other hand, 
\begin{align*}
    \eta(B^{-1}x, (A^{T})^{-1} y) & = (B^{-1}x)^T (A^T)^{-1} y \\
    & = x^T (B^{-1})^T (B^T S^T A^T) (A^T)^{-1} y \\
    & = x^T S y.
\end{align*}
Now since the Smith normal form of a non-degenerate symmetric matrix over $\Z_m$ is diagonal and we may rescale each entry by isomorphisms described above, we obtain an isomorphism with $\A_n(G,\chi)$ as promised.

Now suppose that $k$ is even, and $S$ is invertible and skew-symmetric so that we may assume $S=\begin{pmatrix} 0 &I_0 \\-I_0 &0
\end{pmatrix}$ where $I_0=Id_{(\Z_m)^{k/2}}$ and  $\alpha(x,y)=x^TSy$ the associated bihomomorphism. Notice that $S^2=-I$ and $S^T=-S$.  Then define $\phi(g_i)=(S^{i-1}g)_i$ and $\eta(x,y)=-x^Ty$. To see that $\phi$ is an algebra homomorphism $\A_n(G,\tau)\cong\A_n(G,\chi)$ we compute: $$\alpha(S^ix,S^{i+1}y)=x^T(-S)^i(S)S^{i+1}y=(-1)^ix^TS^{2i+2}y=-x^Ty=\eta(x,y).$$  Since $\phi$ is clearly bijective, we have shown that it is an algebra isomporphism.  Since $\eta$ is symmetric we may use the above to obtain an isomorphism $\A_n(G,\tau)\cong\A_n(G,\chi)$ as promised.

\end{proof}

\subsection{Symmetries of $(\A_n(G,\tau),r)$}\label{r symmetries}

For a fixed $G$ and $\tau$, the set $R$ of $\A(G,\tau)$-YBOs could be quite large: they are determined by the functions $f:G\rightarrow \C$ with $r=\sum_{g\in G}f(g)g\in R$.  To reduce the search space we can make use of various symmetries, identifying function $f$ in the same orbit.  Informally we will say that two $\A_n(G,\tau)$-YBOs $r$ and $s$ are equivalent if $\rho^r,\rho^s:\B_n\rightarrow \A_n(G,\tau)$ have the same image, projectively.

One obvious symmetry comes from the homogeneity of the braid equation  $r_1r_2r_1=r_2r_1r_2$: we can rescale any solution by $z\in\C^\times$ and if our solution is unitary, then we can rescale by $z\in U(1)$.  This corresponds to identifying $f$ and $zf$, since the $\B_n$ images are projectively equivalent.

  We also have rescaling symmetries of the form $g_i\mapsto q^{s(g)}g_i$. Since the straightening relations in $\A_n(G,\tau)$ are homogeneous it is only necessary to check that $\chi:G\rightarrow U(1)$ by $\chi(g)= q^{s(g)}$ is a linear character. This automorphism of $\C[G]$ lifts to an automorphism of $\A_n(G,\tau)$, which carries $r_i=\sum_{g\in G}f(g)g_i$ to $r_i^\chi:=\sum_{g \in G}f(g)q^{s(g)}g_i$ and hence the images are isomorphic. Therefore we can identify the solutions that are in the orbit of $f$ under $f\mapsto q^{s(g)}f$. 
  
  Denote by $\Aut(G,\alpha)$ the group of automorphisms $\psi\in\Aut(G)$ such that $\alpha\circ(\psi\times\psi)=\alpha$. Any $\psi\in\Aut(G,\alpha)$ lifts to an automorphism of $\A_n(G,\tau)$.  For such $\psi\in\Aut(G,\alpha)$ the $\psi(r_i)=\sum_{g\in G}f(g)\psi(g)_i$ for $i=1,2$ will satisfy the braid relation, and hence $\psi^{-1}(r)=\sum_{g\in G}f(\psi^{-1}(g))g$ will also be a $\A(G,\tau)$-YBO.  Moreover, this obviously induces an isomorphism between $\rho^r(\B_n)$ and $\rho^{\psi(r)}(\B_n$, so we will thus identify all solutions in the orbit of $f$ under this symmetry $f\mapsto \psi^*f$ for $\psi\in\Aut(G,\alpha)$.

\subsubsection{Symmetry induced by Inversion}\label{inversion}.    An important special case of symmetry induced by automorphisms is the following: If $G$ is \emph{abelian}, the inversion automorphism $\iota:g\mapsto g^{-1}$ on $G$ lifts to an automorphism of $\A_n(G,\tau)$ by defining $\iota(q)=q$, and $\iota(g_ih_j)=\iota(g_i)\iota(h_j)=g_i^{-1}h_j^{-1}$ on products and extending linearly.  We will carefully check the defining relations are preserved.  Firstly, $$\iota(g_i)\iota(h_{i+1})=g_i^{-1}h_{i+1}^{-1}=(h_{i+1}g_i)^{-1}=(q^{-\alpha(g,h)}g_ih_{i+1})^{-1}=q^{\alpha(g,h)}\iota(h_{i+1})\iota(g_i)$$ so that $\iota(g_ih_{i+1})=q^{\alpha( \iota(g),\iota(h))}\iota(h_{i+1}g_i)$.  Secondly since $G$ is abelian, $$\iota((gh)_i)=\iota(g_ih_i)=g_i^{-1}h_i^{-1}=
    (g^{-1})_i(h^{-1})_i=(g^{-1}h^{-1})_i\stackrel{\star}{=}((gh)^{-1})_i=(\iota(gh))_i,$$ where the equality denoted $\stackrel{\star}{=}$ uses the $G$ abelian assumption.  Finally note that if $g_i$ and $h_j$ commute then so do $\iota(g_i)$ and $\iota(h_j)$.  If $r=\sum_{g\in G}f(g)g$ is an $\A(G,\tau)$-YBO then so is $\iota(r)=\sum_{g\in G}f(g)g^{-1}$, hence $r^{\prime\prime}=\sum_{g\in G}f(g^{-1})g$ is as well.

  If $G$ is abelian there is an additional symmetry of the braid relation $r_1r_2r_1=r_2r_1r_2$ that we may use to show that  $r=\sum_{g\in G}f(g)g$  and $r^\prime:=\sum_{g\in G}\overline{f(g)}g$ have isomorphic images.  The map $\sigma$ on  $\A_3(G,\tau)$ given by $\sigma(g_1)=g_2$ and $\sigma(xy)=\sigma(y)\sigma(x)$ for $x,y\in\A_3(G,\tau)$ and $\sigma(q)=q^{-1}=\overline{q}$ is an anti-automorphism of $\A_3(G,\tau)$, since 
  $$\sigma(g_1h_2)=h_1g_2=q^{\alpha(h,g)}g_2h_1=\sigma(q^{\alpha(g,h)}h_2g_1)$$ and $\sigma((gh)_1)=(gh)_2=g_2h_2=h_2g_2=\sigma(g_1h_1).$ This implies that $r^\prime$ is a $\A(G,\tau)$-YBO since $$r_2^\prime r_1^\prime r_2^\prime =\sigma(r_1r_2r_1)=\sigma(r_2r_1r_2)=r_1^\prime r_2^\prime r_1^\prime.$$  

\section{Case Studies}\label{case studies}
In practice we take the following approach, using symbolic computation software such as Magma and Maple.
\begin{enumerate}
    \item Fix $G$ and $\alpha$, and present the corresponding finitely generated algebra $\A_3(G,\tau)\times \Q(q)[x_g:g\in G]$, using generators and relations, with the $x_g$ being commuting variables.
    \item Define $r_i=\sum_{g\in G} x_gg_i$ for $i=1,2$, and use non-commutative Gr\"obner bases to write $r_1r_2r_1-r_2r_1r_2$ in its normal form i.e., as a polynomial in the $g_1^{(i_1)}g_2^{(i_2)}$ with coefficients in $\Q(q)[x_g:g\in G]$.
    \item Compute a commutative Gr\"obner basis for ideal generated by the the coefficients using pure lexicographic order to find the ideal of solutions.
    \item Use symmetries to describe families of related solutions.
\end{enumerate}
Often we find that there are finitely many solutions, so that we can give a complete description of them.

\subsection{Abelian Groups}

\subsubsection{Prime Cyclic Groups $G=\Z_p$}
We first apply our approach to a well-known case both as a proof of principle and a template for further study.

Let $p\geq 3$ be prime and fix $q$ a primitive $p$th root of unity.  A nontrivial bicharacter $\Z_p\times\Z_p\rightarrow U(1)$ must take values in $\mu_p=\{q^j:0\leq j\leq p-1\}$, so that any bicharacter corresponds to a bihomomorphism $\alpha\in\Hom(\Z_p\times\Z_p,\Z_p)$, which is determined by $\alpha(1,1)$.  Define a bihomomorphism $\alpha: \Z_p\times\Z_p\rightarrow \Z_p$ by $\alpha(1,1)=2$.  The orbit of $\alpha$ under automorphisms of $\Z_p$ give half of all non-trivial bihomomorphisms, since $\alpha(\varphi_k(1),\varphi_k(1))=2k^2$, which is a square modulo $p$ if and only if $2$ is.  Galois symmetry $\psi(q)=q^s$ maps the bicharacter $\tau(x,y)=q^{\alpha(x,y)}$ to $\tau^\psi(x,y)=q^{s\alpha(x,y)}$.  Thus we may assume that our bicharacter is associated with the bihomomorphism $\alpha(x,y)=2xy$.   The reader may wonder why we do not choose $\alpha^\prime(x,y)=xy$ instead--we will see later that this simplifies the form of our $\A_n(\Z_p,\tau)$-YBOs. Indeed,
this choice of $\alpha$ recovers the $\Q(q)$-algebra $\A_n(\Z_p)$ described in the introduction, with generators $u_1,\ldots,u_{n-1}$ satisfying:
$u_iu_{i+1}=q^2u_{i+1}u_i$ and $u_iu_j=u_ju_i$ for $|i-j|>1$ and $u_i^p=1$.
The goal now is to find invertible $\A(\Z_p)$-YBOs $r=\gamma\sum_{j=0}^{p-1}f(j)u^j\in\C[\Z_p]$.

To reduce redundancy we will normalize $f(0)=1$ (the solutions where $f(0)=0$ do not seems to be interesting).  The symmetries of these solutions again come in several forms.  Firstly, since each automorphism of $\Z_p$ that leaves $\alpha$ invariant leads to an automorphism of $\A_n(\Z_p)$ we may identify the corresponding solutions.  For $\alpha(x,y)=2xy$ only inversion $x\rightarrow -x$ leaves $\alpha$ invariant, which means we may freely identify $f$ and $f^\prime(j)=f(-j)$.  We have an additional symmetry in $\A_n(\Z_p)$ given by $u_i\rightarrow q^su_i$, since the first two defining relations are homogeneous and $(q^su_i)^p=u_i^p=1$.  This corresponds to identifying $f$ with $f^s(j):=f(j)q^{js}$.  Finally, complex conjugation is a symmetry of the braid equation $r_1r_2r_1=r_2r_1r_2$, so that we may identify $f$ and its complex conjugate $\overline{f}$.

\subsubsection{The Gaussian Solution} One unitary $\A(\Z_p)$-YBO is the Gaussian solution $r=\frac{1}{\sqrt{p}}\sum_{j=0}^{p-1}q^{j^2}u$, i.e. $f(j)=q^{j^2}$ \cite{GoldschmidtJones,Jones89,GR} and $\gamma=\frac{1}{\sqrt{p}}$.  Complex conjugation gives us the solution $\overline{f}(j)=q^{-j^2}$ and the rescaling symmetry gives us $f^s(j)=q^{j^2+sj}$, giving $2p$ distinct solutions.  

In \cite{GR} it is shown that the braid group representation $\rho_n:\B_n\rightarrow \A_n(\Z_p)$ given by $\sigma_i\rightarrow r_i$ has finite image.  In fact, one has:
\begin{eqnarray*}
 r_iu_{i+1}r_i^{-1}&=&qu_i^{-1}u_{i+1}\\
r_iu_{i-1}r_i^{-1}&=&q^{-1}u_{i-1}u_i,
\end{eqnarray*}
so that the conjugation action on $\A_n(\Z_p)$ provides a homomorphism of $\rho_n(\B_n)$ into monomial matrices, with kernel a subgroup of the center of $\A_n(\Z_p)$.  For $n$ odd the normal form for $\A_n(\Z_p)$ allows one to show that the center consists of scalars, and since any $\rho_n(\beta)$ in the center of $\A_n(\Z_p)$ has determinant a root of unity (under the regular representation of $\A_n(\Z_p)$) the kernel of the conjugation action above is finite, for $n$ odd.  Since $\rho_n(\B_n)\subset\rho_{n+1}(\B_n)$, this is sufficient.

The algebras $\A_n(\Z_p)$ have a \emph{local} representation (see \cite{RW}).  Let $V=\C^p$ and define an operator on $V^{\ot 2}$ by $U(\textbf{e}_i\ot\textbf{e}_j)=q^{j-i}\textbf{e}_{i+1}\ot\textbf{e}_{j+1}$ where $\{\textbf{e}_i\}_{i=0}^{p-1}$ is a basis for $V$ with indices taken modulo $p$.  Then $\Phi_n:u_i\rightarrow (Id_V)^{\ot i-1}\ot U\ot (Id_V)^{\ot n-i-1}$ defines a representation $\A_n(\Z_p)\rightarrow \End(V^{\ot n})$.  In particular $\Phi_2(r)$ is an honest $p^2\times p^2$ YBO.

\begin{example} We use MAGMA \cite{Magma} to work two explicit examples.
First consider the case $G=\Z_3$, and suppose $r=1+au+bu^2$ is a $\A(\Z_3)$-YBO.   All solutions satisfy $a^3=b^3=1$ and $a^2\neq b$, so that there are exactly $6$ distinct solutions (up to rescaling), all of which are obtained from the Gaussian solution via the symmetries described above.  In particular the solutions are all unitary when appropriately normalized.

Similarly for $p=5$,under the additional assumption that $a,b,c,d$ are $5$th roots of unity, we find that there are exactly $10$ non-trivial solutions $r=1+au+bu^2+cu^3+du^4$ (up to rescaling), all of which are obtained from the Gaussian solution via the above symmetries.  These solutions are unitary when appropriately normalized. There are 10 other non-trivial solutions, however, none of them are (projectively) unitary.
\end{example}

\subsection{$G=\Z_p\times \Z_p$}
Let $p$ be an odd prime, and let $G=\Z_p\times\Z_p$.  To classify $\A_n((\Z_p)^2,\tau)$ we first look at orbits of bihomomorphisms $\alpha:(\Z_p)^2\rightarrow \Z_p$.
Such bihomomorphisms are determined by the values on pairs of generators $(1,0),(0,1)$ of $\Z_p\times \Z_p$, encoded in a matrix $A_\alpha:=\begin{pmatrix}a&b\\c&d\end{pmatrix}\in M_2(\Z_p)$, so that $\alpha(x,y)=x^TA_\alpha y$. Under automorphisms $X\in\GL_2(\Z_p)$ of $(\Z_p)^2$ the orbit of $\alpha$ is represented by the matrices $\{X^TA_\alpha X:X\in GL_2(\Z_p)\}$.  From \cite{Waterhouse} we know that there are $p+7$ orbits.  
\begin{example}\label{ex:z3}
The case of $p=3$ can be completely analyzed computationally as follows: from a representative of each of the $10$ orbits of bihomomorphisms $\Z_3\times\Z_3\rightarrow \Z_3$ we use MAGMA \cite{Magma} to search for\emph{non-degenerate, unitary} solutions  $r=\gamma\sum_{i,j}f(i,j)u^iv^j$ to the corresponding $\A(\Z_3\times\Z_3,\tau)$-YBE.  Here by \emph{non-degenerate} we mean that it does not degenerate to the $\Z_3$-case.  The results of these computations are:

\begin{itemize}
    \item Non-degenerate unitary solutions only exists for the classes represented by:\\ $A_1=\begin{pmatrix}2&0\\0&2\end{pmatrix}$,  $A_2=\begin{pmatrix}0&2\\-2&0\end{pmatrix}$ and $A_3=\begin{pmatrix}2&0\\0&-2\end{pmatrix}$.
    \item In all cases, after applying an appropriate symmetry of $(\A_n(\Z_3\times\Z_3,\tau),r)$, the non-degenerate unitary solutions factor as a product of Gaussian $\A_n(\Z_3)$-YBOs, and hence have finite images.
\end{itemize}

\end{example}

From this example we expect that the most interesting ones correspond to non-degenerate symmetric or skew-symmetric bilinear forms on $\Z_p^2$. We also allow ourselves to rescale $\alpha$ by a constant.  Thus we focus on $A_1=\begin{pmatrix}2&0\\0&2\end{pmatrix}$,  $A_2=\begin{pmatrix}0&2\\-2&0\end{pmatrix}$ and $A_3=\begin{pmatrix}2&0\\0&2x\end{pmatrix}$ where $x$ is a non-square modulo $p$.  The appearance of the scalar $2$ is simply for convenience when we make contact with the Gaussian solution.  

We consider each of these cases in turn.  We will distinguish the symmetric cases $A_1,A_3$ by noting that $A_1$ corresponds to an elliptic form, while $A_3$ corresponds to a hyperbolic form.  For $A_\alpha=A_i$ the corresponding algebras $\A_n(\Z_p\times\Z_p,\tau_i)$ have generators 
$u_1,v_1,\ldots,u_{n-1},v_{n-1}$ with the multiplicative group $\langle u_i,v_i\rangle\cong \Z_p\times\Z_p$.  All generators commute except for:

\begin{enumerate}
    \item   For $A_1$: $u_iu_{i\pm1}=q^{\pm2}u_{i\pm 1}u_i$ and $v_iv_{i\pm1}=q^{\pm2}v_{i\pm 1}v_i$.
    \item For $A_2$: $u_iv_{i\pm1}=q^{\pm2}v_{i\pm 1}u_i$ and $v_iu_{i\pm1}=q^{\mp 2}v_{i\pm 1}u_i$.
    \item For $A_3$:  $u_iu_{i\pm1}=q^{\pm2}u_{i\pm 1}u_i$ and $v_iv_{i\pm1}=q^{\pm2x}v_{i\pm 1}v_i$.
\end{enumerate}

We pause to describe the structure of the algebras $\A_n(\Z_m\times\Z_m,\tau_i)$ for arbitrary odd $m$.  Since the monomials in the $u_i,v_i$ form a basis, we see that $\dim_{\Q(q)}\A_n(\Z_m\times\Z_m,\tau_i)=m^{2n-2}.$  

The following proposition explores the structure of $A_n(\Z_m\times\Z_m,\tau_i)$ 
and the subalgebra of fixed points under the automorphism $\iota$ described in Subsection \ref{inversion} given by lifting $u_i\mapsto u_i^{-1},v_i\mapsto v_i^{-1}$ to $\A_n(G,\tau_i)$.  We describe inclusions of algebras in terms of Bratteli diagrams (see \cite{GdlHJ}): generally, to a tower of multi-matrix algebras with common unit $1\in A_1\subset\cdots A_n\subset A_{n+1}\subset\cdots$ we associate a graph with vertices labeled by simple $A_k$-modules $M_{k,i}$ with $d_{k-1,i,j}$ edges between $M_{k,i}$ and $M_{k-1,j}$ if the restriction of $M_{k,i}$ to $A_{k-1}$ contains $M_{k-1,j}$ with multiplicity $d_{k-1,i,j}$.

\begin{prop}

     Let $m$ be odd and $G=\mathbb{Z}_m\times\mathbb{Z}_m$. Consider the algebra $\mathcal{A}_n(G,\tau_i)$ with the twists $\tau_i$ given by $A_i$, $1\leq i\leq 3$. Then \begin{enumerate}
    \item The center of $\A_n(G,\tau_i)$ is 1 dimensional if $n$ is odd and is $m^2$ dimensional if $n$ is even.  Moreover, when $n$ is odd $\A_n(G,\tau_i)\cong M_{m^{n-1}}(\Q(q))$ is simple while for $n$ even $\A_n(G,\tau_i)$  decomposes as a direct sum of $m^2$ simple algebras of dimension $m^{2n-4}$. Moreover, the Bratteli diagram of $\cdots\subset\A_n(G,\tau_i)\subset\cdots$ is given in Figure \ref{fig:bratfora}.
    \item Consider the fixed point subalgebra $\CC_n(G,\tau_i)$ for the automorphism $\iota$ induced by inversion on $\A_n(G,\tau_i)$. Then for $n\geq 3$ odd, $\CC_n(G,\tau_i)$ is a direct sum of two matrix algebras of dimensions $\left(\dfrac{m^{n-1}\pm 1}{2}\right)^2$.   For $n\geq 4$ and even, $\CC_n(G,\tau_i)$ has $\frac{m^2+3}{2}$ simple summands: $\frac{m^2-1}{2}$ of dimension $m^{2n-4}$ and two others of dimensions $\left(\dfrac{m^{n-2}\pm 1}{2}\right)^2$.
Moreover, the Bratteli diagram for $\cdots\subset\CC_n(G,\tau_i)\subset\CC_{n+1}(G,\tau_i)\subset\cdots$ is given by Figure \ref{fig:bratforc}, where the nodes are labelled by the dimensions of the distinct simple modules.

\end{enumerate}
\end{prop}
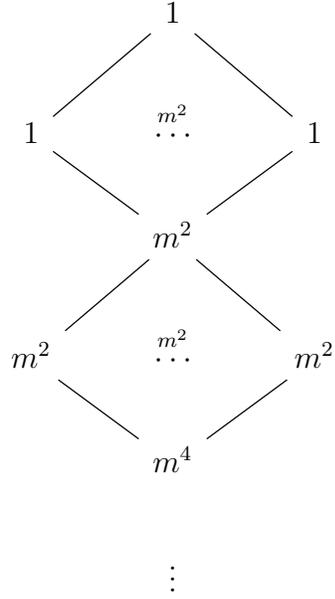
\begin{figure}
    \centering

\begin{tikzcd}
                               & 1 \arrow[ld, no head] \arrow[rd, no head] &                                                                                                  \\
1 \arrow[rd, no head] & \stackrel{m^2}{\cdots}              & 1 \arrow[ld, no head]                         \\
                               & m^2 \arrow[ld, no head]  \arrow[rd, no head] &                           \\
m^2 \arrow[rd, no head] &      \stackrel{m^2}{\cdots}             & m^2 \arrow[ld, no head] \\
                               & m^4                                                                     &                                                                   &     \\ &\vdots     
\end{tikzcd}

    \caption{Bratteli diagram for $\A_n(\Z_m\times\Z_m,\tau_i)$ for $m$ odd.}
    \label{fig:bratfora}
\end{figure}
    \begin{figure}
    \centering

\begin{tikzcd}
                               & 1 \arrow[ld, no head] \arrow[d, "\cdots", no head] \arrow[rd, no head] &                                                                           &                           \\
1 \arrow[rd, no head] & 1 \arrow[d, "\cdots", no head] \arrow[rd, no head]                     & 1 \arrow[ld, no head] \arrow[d, "\cdots"', no head]                     &                           \\
                               & \frac{m^2+1}{2} \arrow[ld, no head] \arrow[d, "\cdots", no head] \arrow[rd, no head] & \frac{m^2-1}{2}\arrow[d, "\cdots"', no head] \arrow[rd, no head] \arrow[ld, no head] &                           \\
\frac{m^2+1}{2} \arrow[rd, no head] & m^2 \arrow[d, "\cdots", no head] \arrow[rd, no head]                     & m^2 \arrow[ld, no head] \arrow[d, "\cdots"', no head]                     & \frac{m^2-1}{2} \arrow[ld, no head] \\
                               & \frac{m^4+1}{2}                                                                      & \frac{m^4-1}{2}                                                                    &\\
                               &             \vdots                                                         &       \vdots                                                                &                       
\end{tikzcd}

    \caption{Bratteli diagram for $\CC_n(\Z_m\times\Z_m,\tau_i)$ for $m$ odd.}
    \label{fig:bratforc}
\end{figure}
    \begin{proof}
    We first note that, by its construction, the isomorphism $\A_n(G, \tau) \cong \A_n(G, \chi)$ of Proposition \ref{smith form} restricts to a bijection on the set $\langle u_j, v_j \rangle$ for each $j$.  It follows that it restricts to a bijection on each $H_j$, hence an isomorphism $\CC_n(G, \tau) \cong \CC_n(G, \chi)$.  Thus, it suffices to prove the proposition for $\tau(x,y) = \chi(x,y) := q^{x^Ty}$.
    
    In this case, we have an isomorphism $\rho:\A_n(\Z_m \times \Z_m), \chi) \to \A_n(\Z_m) \ot \A_n(\Z_m)$ where $\rho(x,y) = q^{xy}$ given by mapping $u_j$ to the $j$-th group generator in the first tensor factor and $v_j$ to the one in the second tensor factor. The structure of $\A_n(\Z_m)$ is well-known (see eg., \cite{jonesstatmech}) from which we can derive the structure of $\A_n(\Z_m \times \Z_m), \chi)$.  Indeed, the Bratteli diagram of $\A_n(\Z_m)$ can be easily worked out: it has the same general shape as in Figure \ref{fig:bratfora} except with $m$ summands for $n$ even and with $\dim\A_n(\Z_m)=m^{n-1}$.
    In particular the simple direct summands of $\A_n(\Z_m \times \Z_m), \chi)$ are of the form $[\A_n(\Z_m)]_i\ot[\A_n(\Z_mo)]_j$  where $[\A_n(\Z_m)]_i$ denotes the $i$th direct summand of $\A_n(\Z_m)$. In particular when $n$ is odd,  $\A_n(\Z_m \times \Z_m), \chi)\cong M_{m^{n-1}}(\Q(q))$ is simple and has a 1-dimensional center whereas for $n$ even $\A_n(\Z_m \times \Z_m), \chi)$ has an $m^2$ dimensional center and decomposes as a direct sum of $m^2$ simple algebras of dimension $m^{2(n-2)}$.

    We compute that $\dim \CC_n(\Z_m \times \Z_m, \chi) = \frac{m^{2(n-1)} +1}{2}$ as follows. Any fixed point of $\iota$ has the form $\sum_{\mathbf{x},\mathbf{y}\in (\Z_m)^{n-1}} c_{\mathbf{x},\mathbf{y}} (a_\mathbf{x}b_\mathbf{y}+a_\mathbf{-x}b_\mathbf{-y})$ where  $a_\mathbf{x}=u_1^{x_1}\cdots u_{n-1}^{x_{n-1}}$ and $b_\mathbf{y}=v_1^{y_1}\cdots v_{n-1}^{y_{n-1}}$. Thus we see that the $\frac{m^{2(n-1)} -1}{2}$ linearly independent binomials of the form $a_\mathbf{x}b_\mathbf{y}+a_\mathbf{-x}b_\mathbf{-y}$ together with $1$ form a basis for  $\CC_n(\Z_m \times \Z_m, \chi)$.

    


    We first consider $n$ odd. Define $d(n)=m^{\frac{n-1}{2}}$. Since $\A_n(\Z_m,\rho)\cong M_{d(n)}(\Q(q))$ is simple of dimension $m^{n-1}$ it has a unique irreducible faithful module $M$ of dimension $d(n)$.  Now the action of the automorphism $\iota$ on $\A_n(\Z_m)$ is an inner automorphism, and hence corresponds to conjugation by some diagonal matrix $J$.  Without loss of generality we may assume $J$ has non-zero entries $\pm 1$ (rescaling if necessary). 
    From \cite{jonesstatmech} we know the structure of $\CC_n(\Z_m,\rho)$: for $n$ odd the restriction of $M$ to $\CC_n(\Z_m,\rho)$ decomposes as a direct sum of two irreducible modules $M_+$ and $M_-$, of dimensions $\frac{d(n)+1}{2}$ and $\frac{d(n)-1}{2}$.  Notice that $\iota(J)=J$ so that $M_+$ and $M_-$ are simply the two eigenspaces of $J$ acting on $M$.   
    
    Now since $\A_n(\Z_m\times\Z_m,\chi)\cong \A_n(\Z_m)\otimes \A_n(\Z_m)$, the action of $\iota$ is (in some basis) conjugation by $J\otimes J\in\A_n(\Z_m)\otimes \A_n(\Z_m)$.  Moreover, as a $\CC_n(\Z_m\times\Z_m,\chi)$ module $M\otimes M$ decomposes as the direct sum of the two $J\ot J$ eigenspaces.  These are precisely $M_+^{\ot 2}\oplus M_-^{\ot 2}$ (the +1 eigenspace) and $M_+\ot M_-\oplus M_-\ot M_+$ (the -1 eigenspace) as vector spaces.  By the double commutant theorem the faithful $\CC_n(\Z_m\times\Z_m,\chi)$-module $M\ot M$ decomposes into a direct sum of two simple modules, which are precisely the $\pm 1$-eigenspaces for $J\ot J$.  Since $\left(\frac{d(n)+1}{2}\right)^2+\left(\frac{d(n)-1}{2}\right)^2=\frac{d(n)^2+1}{2}$ and $2\left(\frac{(d(n)+1)(d(n)-1)}{4}\right)=\frac{d(n)^2-1}{2}$ the dimensions are as stated.
    
    The case of $n$ even is similar, but slightly more complicated since $\A_n(\Z_m)$ is not simple for $n$ even. The Bratteli diagram is found in \cite{jonesstatmech}, and the technique is straight-forward so we sketch it here.  Restricting the simple $\frac{d(n+1)+1}{2}$-dimensional $\CC_{n+1}(\Z_m)$-module $M_+$ to $\CC_n(\Z_m)$ decomposes as $N_+\oplus\bigoplus_{j=1}^{\frac{m-1}{2}}L_j$ and restricting the simple $\frac{d(n+1)-1}{2}$-dimensional $\CC_{n+1}(\Z_m)$-module $M_-$ to $\CC_n(\Z_m)$ decomposes as $N_-\oplus\bigoplus_{j=1}^{\frac{m-1}{2}}L_j$ where $\dim(L_j)=d(n)$ and $\dim(N_{\pm})=\frac{d(n)\pm 1}{2}$.  Indeed, since on $\A_{n}(\Z_m)\cong \bigoplus_{j=0}^{m-1}M_{d(n)}(\Q(q))$ the automorphism  $\iota$ either interchanges pairs of simple summands or restricts to an automorphism on them.  In this case only one simple summand is $\iota$-invariant while $\iota$ induces isomorphisms between the other $\frac{m-1}{2}$ pairs.  The $\iota$-invariant factor then splits as a as direct sum of two $\CC_n(\Z_m)$-modules $N_\pm$ of dimensions $\frac{d(n)\pm 1}{2}$, and we get one distinct $\CC_n(\Z_m)$-module $L_j$ for each $\iota$-orbit of size $2$.  Now we apply the restriction technique above to the two simple $\CC_n(\Z_m\times\Z_m,\chi)$-modules $M_+^{\ot 2}\oplus M_-^{\ot 2}$ and $M_+\ot M_-\oplus M_-\ot M_+$  together with the restriction of $\iota$ to the simple direct summands of $\A_n(\Z_m\times\Z_m,\chi)$ to resolve the count of components, from which the result is derived.

    \end{proof}

  We now return to the problem of finding solutions $$r=\gamma\sum_{0\leq j,k \leq p-1} f(j,k)u^jv^k$$ to the $\A(\Z_p\times\Z_p,\tau_i)$-YBE for $i=1,2$ and $3$.

  We could not find any non-trivial unitary solutions that do not factor as $f(j,k)=f_u(j)f_v(k)$ after applying the symmetries of Subsection \ref{r symmetries}, and can verify computationally that all solutions factor as products of Gaussian-type solutions in the case for $p=3$. Thus we focus on such solutions.  All of the solutions that follow will have finite braid group image when properly normalized to be unitary, which can be easily verified using the finiteness of the Gaussian representation images.  The eigenvalues of $r=\sum_{0\leq j,k\leq p-1}q^{j^2\varepsilon xk^2}u^jv^k$ for $q=e^{2\pi i/p}$ and $\varepsilon=\pm 1$ in any faithful representation of $\A_2(G,\tau_i)$ are: 
  $$\lambda_{\varepsilon,x}(s,t)= \left(\sum_{j}q^{j^2+sj} \right) \left(\sum_k q^{\varepsilon xk^2+tk}\right).$$  Up to an overall normalization factor, the eigenvalues and their multiplicities can be computed using standard Gaussian quadratic form techniques, and only depend on the sign $\pm$ and whether $-1$ and $x$ are squares or non-squares modulo $p$.  We have that the multi-set $[\lambda_{\varepsilon,x}(s,t)]$ has:
  \begin{enumerate}\item $1$ with multiplicity $1$ and $e^{2\pi i j/p}$ with multiplicity $p+1$ for each $1\leq j \leq p-1$ when $(\varepsilon,\legendre{-1}{p},\legendre{x}{p})\in\{(1,-1,1),(-1,-1,-1),(1,1,-1),(-1,1,-1)\}$ and
  \item $1$ with multiplicity $2p-1$ and $e^{2\pi i j/p}$ with multiplicity $p-1$ for each $1\leq j \leq p-1$ when $(\varepsilon,\legendre{-1}{p},\legendre{x}{p})\in\{(1,-1,-1),(-1,-1,1),(1,1,1),(-1,1,1)\}$.
\end{enumerate}
   
\subsubsection{Elliptic Symmetric Case}
First consider the case $A_1$.  
We can deduce some solutions $$t(u,v)=\sum_{0\leq j,k\leq p-1} F(j,k)u^jv^k$$ to the $\A_n(\Z_p\times\Z_p,\tau_i)$-YBE from the Gaussian solutions.  Indeed, if $f,h:\Z_p\rightarrow \C$ and are such that $r(u)=\sum_{j=0}^{p-1}f(j)u^j$ and $s(v)=\sum_{j=0}^{p-1}h(j)v^j$ are solutions to the $\A_n(\Z_p)$-YBE then setting $t(u,v)=r(u)s(v)$ we can easily verify: $$t_1t_2t_1=(r_1s_1)(r_2s_2)(r_1s_1)=(r_1r_2r_1)(s_1s_2s_1)=t_2t_1t_2$$ since $r_i:=r(u_i)$ commutes with $s_i:=s(v_j)$.  If both $f$ and $h$ correspond to Gaussian solutions, we may rescale $u$ and $v$ independently followed by complex conjugation to assume that $f(j)=q^{j^2}$ and $h(k)=q^{\pm k^2}$.  The choice of sign indeed gives two distinct solutions.  The additional symmetry that we have not used comes from the group of $\tau_1$-invariant $G$-automorphism, i.e., $\{X\in\GL_2,(\Z_p): X^TA_1X=A_1\}$, which is a group of order $2(p-1)$ in this case.

\subsubsection{Skew-Symmetric Case}
Next we consider the case $A_2$.  
Suppose that our solution  $$t(u,v)=\sum_{0\leq j,k\leq p-1} F(j,k)u^jv^k$$ factors as $t(u,v)=r(u)s(v)$ where $r(u)=\sum_{j=0}^{p-1}f(j)u^j$ and $s(v)=\sum_{j=0}^{p-1}h(j)v^j$ are solutions to the $\A_n(\Z_p)$-YBE.  Again, setting $r_i=r(u_i)$ and $s_i=s(v_i)$ we observe that $[r_1,r_2]=1$ and $[s_1,s_2]=1$, so that $t_1t_2t_1=r_1s_1r_2s_2r_1s_1=(s_1r_2s_1)(r_1s_2r_1)$.  From this we deduce that we should take $r(u)=s(u)$ and $r(v)=s(v)$, i.e. $h=f$ so that $t(u,v)=r(u)r(v)$.  Now we can use symmetry to choose $f(j)=h(j)=q^{j^2}$.  In this case the group of automorphisms of $\Z_p\times\Z_p$ that preserve $\alpha_2$ is $\SL_2(\Z_p)$, a group of order ($p^2-p)(p+1)$.

\subsubsection{Hyperbolic Symmetric Case}
As the details are similar to the elliptic symmetric case we are content to provide the factored solution $$t(u,v)=\sum_{0\leq j,k\leq p-1}q^{j^2\pm xk^2}u^jv^k.$$  It is an easy exercise to show that this is the unique factorizable solution up to symmetries. The group of automorphisms of $\A(G,\tau_3)$ that preserve  preserve $\alpha_3$ has order $2(p+1)$.

\subsection{Non-commutative cases}
To illustrate our methods for non-abelian groups we first apply them to the case of the symmetric group $S_3$ and the algebra $\A_n(Q_8)$ from the introduction.
\subsubsection{Symmetric group $S_3$}
For $S_3$ the bihomomorphisms $\alpha:S_3\times S_3\rightarrow \Z_m$ are determined by the abelianization $\Z_2\times\Z_2\rightarrow \Z_m$ so that we may take $m=2$. 
In particular we have the following description of $\A_n(S_3,\tau)$ for the non-trivial choice $\alpha((1\;2),(1\;2))=1$.  We take generators $u=(1\;2)$ and $v=(1\;2\;3)$ for $S_3$ and corresponding generators of $\A_n(S_3,\tau)$  $u_1,v_1,\ldots,u_{n-1},v_{n-1}$ with relations:
\begin{itemize}
    \item $u_iv_i=v_i^2u_i$ and $u_i^2=v_i^3=1$ ($S_3$ relations) and
    \item $u_iu_{i+1}=-u_{i+1}u_i$ and $v_iv_j=v_jv_i$ for all $i,j$, and 
    \item $u_iv_j=v_ju_i$ for $i\neq j$.
\end{itemize}
We seek (invertible) solutions $r=\gamma(1+au+bv+cv^2+duv+euv^2)\in\C[S_3]$ to the $\A(S_3,\tau)$-YBE, where $\gamma$ is a normalization factor chosen to give $r$ finite order.
Appendix I contains the details of the computation, the upshot of which is that $b=c=0$ is a consequence of invertibility and to have solutions $r$ that are unitary with respect to the standard $*$-operation we should take $\gamma=\frac{1}{1+\im}$ and $(a,d,e)=(\im x,\im y,\im z)$ with $(x,y,z)\in\R^3$ on the intersection of the surface given by $xy+xz+yz=0$ with the unit sphere $x^2+y^2+z^2=1$.  Since $(x+y+z)^2=1$ modulo the ideal generated by these two polynomials we conclude that the solutions are the points on the intersection of the two planes $(x+y+z)=\pm 1$ with the unit sphere. 

In all cases we find that $r^4=1$, with eigenvalues $1,-\im$.  This suggests that this representation is related to the Ising theory, see \cite{FRW}.

\subsubsection{Quaternionic Algebra $\A_n(Q_8)$}
Recall the algebra $\A_n(Q_8)$ described in the introduction, generated by $u_i,v_i$ satisfying:
\begin{enumerate}
\item $u_i^2=v_i^2=-1$ for all $i$,
\item $[u_i,v_j]=-1$ if $|i-j|<2$,
\item $[u_i,u_j]=[v_i,v_j]=1$,
\item $[u_i,v_j]=1$ if $|i-j|\geq2$.

\end{enumerate}
From the relations one deduces that for each $i$ the pair $u_i,v_i$ generates a group isomorphic to $Q_8$.  Notice, however, that $\langle u_i,v_i\rangle \cap \langle u_i,v_i\rangle=\{\pm 1\}$ so that $\A_n(Q_8)$ is not a twisted tensor product of group algebras: indeed it is not $\C[Q_8]^{\ot n-1}$ as a vector space.  The algebra is closely related to group algebras, in at least two ways.  Firstly, suppose that $Q_8=\langle u,v\rangle$ where $uv=zvu$ with $u^2=v^2=z$ central of order $2$.  Then we may define the quotient $\mathcal{T}=\C[Q_8]/\langle z+1\rangle$ and then $\A(Q_8)$ above is a twisted tensor product of $n-1$ copies of $\T$ with the tensor product twist given as above, determined by $\tau(u,v)=-1$ and $\tau(u,u)=\tau(v,v)=1$ since $u_iv_{i+1}=-1v_{i+1}u_i$.

Alternatively, we can consider the twisted group algebra $\C^\nu[\Z_2\times\Z_2]$ associated with the cocycle $\nu\in Z^2(\Z_2\times \Z_2,U(1))$ defined by:
\begin{itemize}
    \item $\nu((1,0),(0,1))=-\nu((0,1),(1,0))=1$
    \item $\nu((1,0),(1,0))=\nu((0,1),(0,1))=-1$
\end{itemize}
with multiplication in $\C^\nu[\Z_2\times\Z_2]$ given by $g\star_\nu h=\nu(g,h)gh$ for $g,h\in\Z_2\times\Z_2$.  Then
$$\A_n(Q_8)=\C^\nu[\Z_2\times\Z_2]\ot_\tau\C^\nu[\Z_2\times\Z_2]\ot_\tau\cdots\ot_\tau \C^\nu[\Z_2\times\Z_2],$$ where $\tau$ is the twisting corresponding to the relations above.

We look for $\A(Q_8)$-YBOs of the form $r=1+au+bv+cuv$.  We find $8$ non-trivial solutions namely $a,b,c\in\{\pm \frac{1}{2}\}$, normalized to a unitary solution.  These are all related by symmetry since we may rescale $a\rightarrow -a$ and $b\rightarrow -b$ independently, permute the $a,b,c$ freely using the fact that $\tau$ is invariant under permuting $u,v,uv$ and inversion corresponds to simultaneously changing all signs.  The Magma code is found in Appendix A.

\section{Categorical Connections}  \label{categorical connections}
The class of weakly integral modular categories, i.e. those for which $\FPdim(\CC)\in\Z$ is not well-understood.  However, a long-standing question \cite[Question 2]{ENO2} asks if the class of weakly integral fusion categories coincides with the class of weakly group-theoretical fusion categories, i.e., those that are Morita equivalent to a nilpotent fusion category.  Recently Natale \cite{Nat} proved that any weakly group-theoretical modular category is a $G$-gauging of either a pointed modular category (all simple objects are invertible) or a Deligne product of a pointed modular category and an Ising-type modular category \cite[Appendix B]{DGNO}.  These latter categories are well-known to have property $F$, which reduces the verification of the property $F$ conjecture for weakly integral braided fusion categories to verifying that $G$-gauging preserves property $F$ and that weak integrality is equivalent to weak group-theoreticity.   

The difficulty with verifying property $F$ for a given category is that one rarely has a sufficiently explicit description of the braid group representations $\rho_X$ associated with an object $X\in\CC$.  The braiding $c_{X,X}$ on $X\otimes X$ provides a map $\C\B_n\rightarrow \End(X^{\ot n})$ which then acts on each $\Hom(Y,X^{\ot n})$ for simple objects $Y$ by composition, but an explicit basis for $\Hom(Y,X^{\ot n})$ is lacking.  In all cases where the property $F$ conjecture has been verified for a weakly integral braided fusion category \cite{ERW,NR,RW,jones86,FLW} the first step is a concrete description of the centralizer algebras $\End(X^{\ot n})$, and the corresponding modules $\Hom(Y,X^{\ot n})$ which are obtained by studying a specific realization of $\CC$, as a subquotient category of representations of a quantum group, for example.  From this description one extracts a sufficiently explicit $\B_n$ representation to facilitate the verification of property $F$.

One approach to a uniform proof of (one direction of) the property $F$ conjecture is to understand the connection between the centralizer algebras of pointed modular categories and those of its $G$-gaugings.  Pointed modular categories are in one to one correspondence with metric groups, i.e., pairs $(A,Q)$ where $A$ is a finite abelian group and $Q$ is a non-degenerate quadratic form on $A$. We denote by $\CC(A,Q)$ the corresponding modular category.  Pointed modular categories and their products with Ising-type categories are well-known to have property $F$ \cite{NR}.  If it could be proved that $G$-gauging preserves property $F$ then we could reduce this direction of the property $F$ conjecture to \cite[Question 2]{ENO2}.

For a general mathematical reference on $G$-gaugings, see \cite{cuiplavniketal}, the notation of which we will adopt here. Let $A$ be an abelian group, and $Q$ a non-degenerate quadratic form on $A$, and $\CC(A,Q)$ the corresponding pointed modular category, with twists given by $\theta_a=Q(a)$ and braiding by $c_{a,b}=\beta(a,b)\sigma$ where $\sigma$ is the usual flip map and $\beta(a,b):=Q(a+b)-Q(a)-Q(b)$.  A $G$ symmetry of $\CC(A,Q)$ is a group homomorphism $\rho:G\rightarrow \Aut_\ot^{br}(\CC(A,Q))\cong O(A,Q)$. Provided certain cohomological obstructions vanish one may construct (potentially several) modular categories by \emph{gauging} the $G$ symmetry. 
In the case of an elementary abelian $p$-group for $p$ an odd prime all of the obstructions vanish by \cite[Theorem 6.1]{EtGa}.

We expect there to be a connection between the algebras $\A_n(G,\tau)$ described above and the $H$-gaugings of pointed modular categories, i.e., categories $\CC(G,Q)^{\times, H}_H$ where $H\subset \Aut_\ot^{br}(\CC(G,Q))$.  Indeed, in the case $G=\Z_p$ and $H=\Z_2$ acting by inversion these categories are called $p$-\emph{metaplectic} and we we have the following, using results of \cite{RWQT,GRRTohoku,jonesstatmech,ACRW} and some careful adjustment of parameters:
\begin{theorem}
Let $\Z_2$ act on $\CC:=\CC(\Z_p,Q)$ by inversion, and let $\D=\CC^{\times,\Z_2}_{\Z_2}$ be any of the corresponding gaugings, and $X$ a simple object of dimension $\sqrt{p}$.  Then $$\End(X^{\ot n})\cong \langle u_1+u_1^{-1},\ldots,u_{n-1}+u_{n-1}^{-1}\rangle\subset\A_n(\Z_p).$$
\end{theorem}
In fact, this result is key to verifying the property $F$ conjecture for $p$-metaplectic categories.

A similar relationship exists between a $\Z_3$-gauging of the so-called \emph{three fermion} theory $\CC(\Z_2\times\Z_2,Q)$ where $Q(x)=-1$ for $x\neq (0,0)$ and the algebra $\A_n(Q_8)$ described above. In this case $ \CC(\Z_2\times \Z_2,Q)^{\times,\Z_3}_{\Z_3}\cong SU(3)_3$ for one choice of $\Z_3$-gauging, where the action of $\Z_3$ at the level of object is given by cyclic permutation of the three non-trivial simple objects (see \cite{cuiplavniketal}).  Now for a generating $2$-dimensional object $X$ it is shown in \cite{RQT} that the subalgebra $\CC_n(Q_8)$ of $\A_n(Q_8)$ generated by $(u_i+v_i+u_iv_i)$ for $1\leq i\leq n-1$ is isomorphic to $\End(X^{\ot n})$, which is also isomorphic to a quotient of the Hecke algebra specialization $\cH_n(3,6)$.  The reader will also notice that the $\Z_3$ action on $Q_8$ given by cyclic permutation of $u,v$ and $uv$ lifts to an automorphism of $\A_n(Q_8)$ and $\CC_n(Q_8)$ is the fixed point subalgebra.  Finally, we remark that the image of the braid group representation on $\End(X^{\ot n})$ is finite--it factors through the representations found above: $\sigma_i\mapsto (1+u_i+v_i+u_iv_i)$.

This inspires the following:
\begin{principle}
If $G\subset \Aut(A,Q)$ is a gaugeable action on $\CC(A,Q)$ then there is a (quotient of an) iterated twisted tensor product $\A_n(A,\tau)$ of $\C[A]$ and an object $X\in\CC(A,Q)^{\times,G}_G$ so that $\End(X^{\ot n})$ is isomorphic to the fixed point subalgebra $\CC_n(A,\tau)$ of the automorphism induced by the action of $G$ on $A$.  Moreover there is a $\A(A,\tau)$-YBO $r$ supported in $\CC_n(A,\tau)$ such that the $\B_n$ representation on $\End(X^{\otimes n})$ factors through the $B_n$ representation defined by $r$.
\end{principle}

We do not have a general proof of this principle for all groups.  In the case of $\Z_p\times\Z_p$ with $\Z_2$ acting by inversion we give some compelling evidence for this principle.

Now suppose that 
 $|A|=m=2k+1$ is odd and $\rho:\Z_2\rightarrow \Aut_\ot^{br}(\CC(A,Q))$ is the action by inversion.  The $\Z_2$-extensions are Tambara-Yamamgami categories $TY(A,\chi,\pm)$ \cite{TY}, and their equivariantizations are found in \cite{GNN} (see also \cite{IzumiII}).
There are two distinct $\Z_2$-gaugings $\mathcal{D}_\pm := \CC(A,Q)_{\mathbb{Z}_2}^{\times, \mathbb{Z}_2}$ of the inversion action $\rho$.  Each modular category $\mathcal{D}_\pm$ has dimension $4 \left| A \right|$.  It has the following simple objects:
\begin{enumerate}
\item[] two invertible objects, $\mathbf{1} =X_+$ and $X_-$,
\item[] $\frac{m-1}{2}$ two-dimensional objects
$Y_a, \, a\in A-\{0\}$ (with $Y_{-a} =Y_a$)
\item[] two $\sqrt{m}$-dimensional objects $Z_l$, $l \in \mathbb{Z}_2$.
\end{enumerate}

The fusion rules of $\mathcal{D}_\pm$ are given by:
\begin{align*}
 X_- \ot X_- &= X_+,     & X_\pm \ot Y_a &= Y_a,    & X_+ \ot Z_l &=Z_l,  \\
 X_- \ot Z_l &= Z_{l+1}, & Y_a \ot Y_b &= Y_{a+b} \oplus Y_{a-b}, & Y_a \ot Y_a &= X_+\oplus X_- \oplus Y_{2a},\\
 Y_a\ot Z_l &= Z_0 \oplus Z_1, & Z_l \ot Z_l  &= X_+ \oplus \left( \oplus_a\, Y_a\right), & Z_l \ot Z_{l+1}  &= X_- \oplus \left(\oplus_a\, Y_a \right),
\end{align*}
where $a,\, b\in A \; (a\neq b)$ and $l\in \mathbb{Z}_2$.
All objects of $\mathcal{D}_\pm$ are self-dual.  Here the addition $a+b$ takes place in $A$.  We see that $X_-$ must be a boson, in the sense that the subcategory $\langle X_-\rangle\cong\Rep(\Z_2)$ as a braided fusion category.  Indeed, as $\D_\pm$ is a non-degenerate braided fusion category it is faithfully $\Z_2$-graded with the trivial component having the $\frac{m+1}{2}$ simple objects $Y_a,X_{\pm}$, and non-trivial component having the two simple objects $Z_l$.  

In particular the algebras $\End(Z_0^{\ot n})\subset\End(Z_0^{\ot n+1})$ have the Bratteli diagram of Figure \ref{fig:brat}, where we have labeled the objects $Y_a$ by an arbitrary choice $Y_i$ for $1\leq i\leq k$.

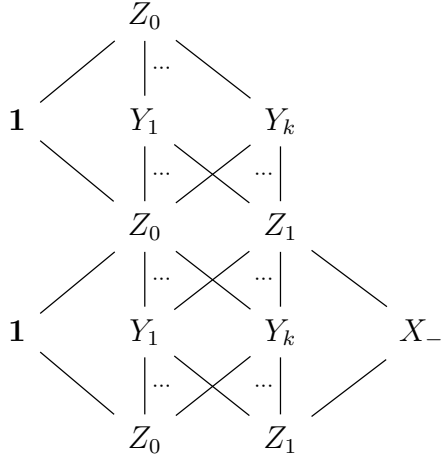
\begin{figure}
    \centering

\begin{tikzcd}
                               & Z_0 \arrow[ld, no head] \arrow[d, "\cdots", no head] \arrow[rd, no head] &                                                                           &                           \\
\textbf{1} \arrow[rd, no head] & Y_1 \arrow[d, "\cdots", no head] \arrow[rd, no head]                     & Y_k \arrow[ld, no head] \arrow[d, "\cdots"', no head]                     &                           \\
                               & Z_0 \arrow[ld, no head] \arrow[d, "\cdots", no head] \arrow[rd, no head] & Z_1 \arrow[d, "\cdots"', no head] \arrow[rd, no head] \arrow[ld, no head] &                           \\
\textbf{1} \arrow[rd, no head] & Y_1 \arrow[d, "\cdots", no head] \arrow[rd, no head]                     & Y_k \arrow[ld, no head] \arrow[d, "\cdots"', no head]                     & X_{-} \arrow[ld, no head] \\
                               & Z_0                                                                      & Z_1                                                                       &                          
\end{tikzcd}

    \caption{Bratteli diagram for $\CC(A,Q)^{\times,\Z_2}_{\Z_2}$ for $|A|$ odd.}
    \label{fig:brat}
\end{figure}

The categories $\D_\pm$ described above for the group $G=\Z_p\times\Z_p$ were explored in \cite{GNN}, and found to be non-group-theoretical in one case and group-theoretical in the other.  For the case $p=3$, the group-theoretical cases are equivalent to $\Rep(D^\om S_3)$  where $\om$ is a $3$-cocycle on $S_3$.  Up to equivalence there is one non-trivial choice for $\om$.

We expect that:
\begin{enumerate}
    \item $\End(Z^{\ot n}) \cong C_n(\Z_p\times\Z_p,\tau)$ for some choice of $\tau$.
    \item Under the above isomorphisms the image of the braid group generators are described by the $\A(\Z_p\times\Z_p,\tau)$-YBOs determined above.
\end{enumerate}

The two pieces of evidence are as follows:

\begin{enumerate}
    \item The Bratteli diagrams for  $\End(Z_0^{\ot n})$ and $\CC_n(\Z_p\times\Z_p,\tau_i)$ coincide and 
    \item The eigenvalue profile of $c_{Z_0,Z_0}$ and $\sum_{j,k}q^{j^2\pm xk^2}u^jv^k$ coincide, for some choice of $\pm x$ where $x$ is either $1$ or any non-square modulo $p$.
\end{enumerate}

The $S$ and $T$ matrices of all 4 of these categories are given in \cite{GNN}, as they are equivalent to $\Z_2$-equivariantizations of Tambara-Yamagami categories.  From \cite[Prop. 2.3]{BDGRTW} we may deduce the eigenvalues of the braiding for the object $Z_0$ of dimension $p$.

\subsubsection{A special case: $p=3$} Let $q=e^{2\pi \im/3}$.
The two group-theoretical categories $\Rep(D^\om S_3)$ can be obtained by gauging the $\Z_2$ inversion symmetry on $\CC(\Z_3\times\Z_3,Q_1)$ where $Q_2(x,y)=q^{x^2-y^2}$ is hyperbolic.  For the elliptic quadratic form $Q_2(x,y)=q^{x^2+y^2}$ the two inequivalent $\Z_2$-gaugings are non-group-theoretical.  Each of these categories can be tensor generated by a simple object $Z$ of dimension $3$.  The two group theoretical-categories $\Rep(D^\om G)$ have property $F$ (\cite{ERW}), but it is currently open whether the non-group-theoretical cases have property $F$.

On the other hand, we can have completely determined all unitary solutions to the $\A(\Z_3\times\Z_3,\tau)$-YBE for the bicharacters $\tau$ associated with the $3$ matrices $A_1,A_2$ and $A_3$, up to the usual symmetries in Example \ref{ex:z3}.

One more piece of circumstantial evidence is that the results of \cite{GRRTohoku} show that the braid group representations associated with a modular category only nominally depend on the finer structures such as the associativity constraints: for odd primes $p$, the images of the braid group representations for $p$-metaplectic modular categories are projectively equivalent.  Since the property $F$ conjecture depends only on the dimensions of objects, which are determined by fusion rules, it would perhaps not be so surprising if the fusion rules essentially determine the braid group images.  A related result in \cite{nikbraid} implies that the images of the braid group representations associated with different modular categories with the same underlying fusion category are either all finite or all infinite.

\section{Conclusions}\label{conclusions}
In this paper we have unified some explicit constructions of braid group representations that come from finite groups in a fairly direct way.
We have also provided strong evidence that twisted tensor products of group algebras simplify the analysis of gaugings of pointed modular categories.  In particular, the data describing $\A(A,\tau)$-YBOs, simply a function on $A$, is much simpler than the construction of the $R$-matrices of a gauged modular tensor category.  However, beyond the Gaussian case, the connection between the two braid group representations remains at the level of Bratteli diagrams and eigenvalues.    

It would be of great interest to formulate precise intertwining operators between braid group representations in centralizer algebras of gaugings of pointed modular tensor categories and those from $\A(A,\tau)$-YBOs.  This was accomplished with great difficulty in the Gaussian case \cite{RW}.  Ideally, we would like to find a uniform framework generalizing this construction to all gaugings of pointed modular tensor categories.

\input{magma}

\bibliographystyle{plain}
\bibliography{Reference.bib}

\end{document}

%% file: magma.tex
\section*{Appendix I: computations for $G=S_3$ and $\A_n(Q_8)$.}
In what follows we provide some details classifying solutions to the $\A(S_3,\tau)$ and $\A(Q_8)$-YBE.
\subsection*{Symmetric group $S_3$}
We let $u,v$ be the generators for $S_3$ with $u^2=v^3=1$ and $uvu=v^2$.  For example we could take $u=(1\;2)$ and $v=(1\;2\;3)$
By the theory above, we initialize with the following MAGMA code to find conditions on $a,b,c,d,e\in\C$ so that $r=1+au+bv+cv^2+duv+euv^2$ is an $\A(S_3,\tau)$-YBO.
\begin{verbatim}
    
F<w>:=CyclotomicField(4);
R<u1,v1,v2,u2,a,b,c,d,e>:=FreeAlgebra(F,9);
f:=function(x,y,z)
return x*y-z*y*x;
end function;
X:=[u1,v1,u2,v2,a,b,c,d,e];
B:=[u1^2-1,v1^3-1,u2^2-1,v2^3-1,u1*v1-v1^2*u1,u2*v2-v2^2*u2] cat
[f(u1,u2,-1),f(v1,v2,1),f(u1,v2,1),f(u2,v1,1)] cat
[f(a,x,1): x in X] cat 
[f(c,x,1): x in X] cat
[f(b,x,1): x in X] cat
[f(e,x,1): x in X] cat
[f(d,x,1): x in X];
I:=ideal<R|B>;

R1:= 1+a*u1+b*v1+c*v1^2+d*u1*v1+e*u1*v1^2;
R2:=1+a*u2+b*v2+c*v2^2+d*u2*v2+e*u2*v2^2;
NormalForm(R1*R2*R1-R2*R1*R2,I);

\end{verbatim}
The ideal of solutions is generated by the coefficients of the monomials in $u_i,v_j$.  We enforce invertibility of $r$ by assuming the determinant of the image of $r$ under the faithful $S_3$ representation on $\C^3$ is non-zero.  The output of the Gr\"obner basis is the following set of polynomials: 
$$\{c, b, e(a^2+d^2+e^2+1), 
ad+ae+de, a^3+a^2e+2ae^2+de^2+e^3+a+e, 
-a^2e+ae^2+d^3+2de^2+d\}$$
Notice that $c=b=0$, in all cases. If $e=0$ then $ad=0$, and $a^3+a=d^3+d=0$, which are degenerate solutions of the form $1+xu$ that can be obtained from $\Z_2$ (see \cite{FRW}).

If $e\neq 0$ we find that $e$ is a free parameter, and the following code shows that we may normalize to get $r^4=1$.  There is a 1-parameter family of solutions for $(a,d,e)$.  Moreover one sees that if we require a unitary solution each of $a,d,e$ should be pure imaginary, and consequently the equation $a^2+d^2+e^2+1=0$ implies that $(a/\im,d/\im,e/\im)$ is a point on the unit sphere.   Geometrically this is the intersection of the unit sphere with the surface given by $xy+xz+yz=0$.

\begin{verbatim}
F<w>:=CyclotomicField(4);
R<u1,v1,v2,u2,a,d,e>:=FreeAlgebra(F,7);
f:=function(x,y,z)
return x*y-z*y*x;
end function;
X:=[u1,v1,u2,v2,a,d,e];
B:=[u1^2-1,v1^3-1,u2^2-1,v2^3-1,u1*v1-v1^2*u1,u2*v2-v^2*u2] cat
[f(u1,u2,-1),f(v1,v2,1),f(u1,v2,1),f(u2,v1,1)] cat
[f(a,x,1): x in X] cat 
[f(e,x,1): x in X] cat
[f(d,x,1): x in X] cat [a^4+2*a^3*e+3*a^2*e^2+2*a*e^3+e^4+a^2+2*a*e+e^2, 
a^3+a^2*d+2*a^2*e+a*e^2+e^3+a+e, a*d+a*e+d*e, (a^2+d^2+e^2+1)];
I:=ideal<R|B>;

R1:= (1+a*u1+d*u1*v1+e*u1*v1^2)/(1+w);
R1i:=(1-(a*u1+d*u1*v1+e*u1*v1^2))/(1-w);
R2:=(1+a*u2+d*u2*v2+e*u2*v2^2)/(1+w);
NormalForm(R1*R2*R1-R2*R1*R2,I);
NormalForm(R1^4,I);
NormalForm(R1i*R1,I);
\end{verbatim}

\subsection*{Quaterionic Algebra $\A_n(Q_8)$}
For the case of the algebra $\A_n(Q_8)$ we use Magma to classify $\A(Q_8)$-YBOs.  The following is the final code, where the last polynomial relations are the coefficients obtained from an initial run of the normal form command on an initial run (i.e., without the last set of relations).  One finds that the non-trivial solutions for $(a,b,c)$ are all $\pm 1$, so that if we want unitary solutions, the inverse of {R1} is of the form given as {R1i} since $u^*=u^{-1}=-u$ etc.  We conclude that all unitary solutions are equivalent to the choice $(a,b,c)=(1,1,1)$.
\begin{verbatim}
F<w>:=CyclotomicField(12);
R<u1,v1,v2,u2,a,b,c>:=FreeAlgebra(F,7);
f:=function(x,y,z)
return x*y-z*y*x;
end function;
g:=function(a)
return a^2+1;
end function;
X:=[u1,v1,v2,u2,a,b,c];
Y:=[u1,v1,v2,u2];
B:=[g(x):x in Y] cat
[f(u1,v1,-1),f(u2,v2,-1)] cat
[f(u1,u2,1),f(v1,v2,1),f(u1,v2,-1),f(u2,v1,-1)]  
cat [f(a,x,1): x in X]
cat [f(b,x,1): x in X]
cat [f(c,x,1): x in X]
cat [-2*b*a^2+ 2*c^2*b,2*b^2*a-2*c^2*a,
c*a^2+c*b^2-c^3-c,4*c*b^2- 2*c^3-2*c,
a^3-2*c^2*a+a,b^3-2*c^2*b+b];
I:=ideal<R|B>;

R1:=(1+a*u1+b*v1+c*u1*v1)/2;
R2:=(1+a*u2+b*v2+c*u2*v2)/2;
R1i:=1/2-(a*u1+b*v1+c*u1*v1)/2;
NormalForm(R1*R2*R1-R2*R1*R2,I);
NormalForm(R1*R1i,I);

\end{verbatim}